\documentclass[preprint,12pt]{elsarticle}
\pdfoutput=1
\usepackage{ifpdf}
\usepackage{bm}
\usepackage{mathdots}
\usepackage{soul, color}
\usepackage{amsmath}
\usepackage{mathrsfs}
\usepackage{graphicx}
\usepackage{subfig}
\usepackage{epstopdf}
\usepackage{epsfig}

\newtheorem{theorem}{Theorem}[section]
\newtheorem{lemma}[theorem]{Lemma}

\newtheorem{corollary}[theorem]{Corollary}

\newenvironment{proof}[1][Proof]{\begin{trivlist}
\item[\hskip \labelsep {\bfseries #1}]}{\end{trivlist}}
\newenvironment{definition}[1][Definition]{\begin{trivlist}
\item[\hskip \labelsep {\bfseries #1}]}{\end{trivlist}}

\newenvironment{remark}[1][Remark]{\begin{trivlist}
\item[\hskip \labelsep {\bfseries #1}]}{\end{trivlist}}

\usepackage{amssymb}

\journal{Journal}

\begin{document}

\begin{frontmatter}



\title{Optimal preconditioners for systems defined by functions of Toeplitz matrices}

\author{Sean Hon\fnref{label1}}

\address{Mathematical Institute, University of Oxford, Radcliffe Observatory Quarter, Oxford, OX2 6GG,
United Kingdom}
\ead{hon@maths.ox.ac.uk}
\fntext[label1]{The research of the author was partially supported by Croucher Foundation.}

\begin{abstract}
We propose several circulant preconditioners for systems defined by some functions $g$ of Toeplitz matrices $A_n$. In this paper we are interested in solving $g(A_n)\mathbf{x}=\mathbf{b}$ by the preconditioned conjugate method or the preconditioned minimal residual method, namely in the cases when $g(z)$ are the functions $e^{z}$, $\sin{z}$ and $\cos{z}$. Numerical results are given to show the effectiveness of the proposed preconditioners.
\end{abstract}

\begin{keyword}
Toeplitz matrices \sep Functions of matrices \sep Circulant preconditioners \sep PCG \sep PMINRES  

\MSC[] 65F08 \sep 65F10 \sep 65F15 \sep 15A16 \sep 15B05
\end{keyword}

\end{frontmatter}



\section{Introduction}

Motivated by \cite{Jin2014224} in which the authors proposed some optimal preconditioners for certain functions of general matrices, we show that $g(c(A_n))$, where $c(A_n)$ is the optimal circulant preconditioner for $A_n$ first proposed in \cite{doi:10.1137/0909051}, is an effective preconditioner for $g(A_n)$. Specifically we are interested in the cases when $g(z)$ are the trigonometric functions $e^{z}$, $\sin{z}$ and $\cos{z}$. 

A crucial application of $e^{A_n}$, for example, arises from the discretisation of integro-differential equations with a shift-invariant kernel \cite{2016arXiv160701733K}. Solving those equations is often required in areas such as the option pricing \cite{duffy2013finite, SACHS20081687}. Related work on computing the exponential of a block Toeplitz matrix arising in approximations of Markovian fluid queues can also be found in \cite{BINI2016387}.

Over the past few decades, preconditioning for Toeplitz matrices with circulant matrices has been extensively studied. Strang \cite{Strang:1986:PTM:12330.12335} and Olkin \cite{OLKIN1986} were the first to propose using circulant matrices as preconditioners in this context. Theoretical results that guarantee fast convergence with circulant preconditioners were later given by \cite{doi:10.1137/0910009}. Other circulant preconditioners such as optimal preconditioners \cite{doi:10.1137/0909051}, Huckle's preconditioners \cite{doi:10.1137/0613048} and superoptimal preconditioners \cite{doi:10.1137/0613030} were then developed for certain classes of Hermitian and positive definite Toeplitz matrices generated by positive functions $f$. The restriction on $f$ was later relaxed for example in \cite{Chan:1996:CGM:240441.240445, Capizzano1999SuperlinearPM, doi:10.1137/0916041}. Some work had also been done on preconditioning for Hermitian indefinite Toeplitz systems \cite{doi:10.1137/S0895479899362521}, non-Hermitian Toeplitz systems \cite{MR2125632} and nonsymmetric Toeplitz systems \cite{POTTS1998265}. For references on the development of preconditioning of Toeplitz matrices we refer to \cite{MR2108963, MR2376196}.

Throughout this work we consider the case when $f$ is a $2\pi$-periodic continuous complex-valued function as analysed in \cite{doi:10.1137/0730062}, thus the corresponding Toeplitz matrix $A_n[f]$ is in general complex and non-Hermitian. Consequently $g(A_n[f])$ is also a non-Hermitian complex matrix. We let $c_n[f]$ be the optimal circulant preconditioner \cite{doi:10.1137/0909051} derived from $A_n[f]$. Using $g(c_n[f])$ as the preconditioner, we can then apply CG to the normal equations system  \[ [g(c_n[f])^{-1} g(A_n[f])]^{*}[g(c_n[f])^{-1} g(A_n[f])]\mathbf{x}=[g(c_n[f])^{-1} g(A_n[f])]^{*}g(c_n[f])^{-1}\mathbf{b}. \] We also consider the special case in which we can use MINRES for the Hermitian indefinite $g(A_n[f])$ with the preconditioner $g(c_n[f])$.

Given a circulant matrix $C_n$, we remark that $g(C_n)$ is also a circulant matrix. By the diagonalisation ${C_n=F_n^{*}\Lambda_n F_n}$, where $F_n$ is a Fourier matrix \cite{MR543191} in which the entries are given by $[F_n]_{jk}=\frac{1}{\sqrt{n}}e^{-2\pi \mathbf{i} j k/n}$ with $j,k=0,1,\dots,n-1$, we have \[ g(C_n) = F_n^{*} g(\Lambda)F_n. \] Therefore, for any vector $\mathbf{d}$ the product $g(C_n)^{-1}\mathbf{d}$ can by efficiently computed by several Fast Fourier Transforms (FFTs) in $\mathcal{O}(n\log{n})$ operations \cite{MR735963}.

It must be noted however that fast matrix vector multiplication with the matrix $g(A_n[f])$ is not readily archived by circulant embedding. For $e^{A_n[f]}$, the matrix vector multiplication can be computed efficiently for example by a fast algorithm in $\mathcal{O}(n\log{n})$ operations \cite{doi:10.1137/090758064}.

This paper is outlined as follows. In section $2$ we first provide some lemmas on bounds for the spectra of $A_n[f]$ and $c_n[f]$. In section $3$ we also give several lemmas on functions of matrices. In section $4$ we provide the main results on the preconditioned matrix $g(c_n[f])^{-1}g(A_n[f])$. In section 5 we present numerical results to demonstrate the effectiveness of the proposed preconditioners.

\section{Spectra of $c_{n}[f]$ and $A_{n}[f]$ }

Denote by $\mathcal{C}_{2\pi}$ the Banach space of all $2\pi$-periodic continuous complex-valued functions equipped with the supremum norm $\|\cdot\|_{\infty}$. For all $f \in \mathcal{C}_{2\pi}$, we let \[ a_{k}=\frac{1}{2\pi} \int_{-\pi} ^{\pi}f(\theta) e^{-\mathbf{i} k \theta } \,d\theta,\quad k=0,\pm1,\pm2,\dots,  \] be the Fourier coefficients of $f$. Let $A_{n}[f]$ be the $n$-by-$n$ complex Toeplitz matrix with the $(j,k)$-th entry given by $a_{j-k}$. The function $f$ is called the \emph{generating function} of the matrix $A_n[f]$.
We then have
 \[ A_{n}[f]=\begin{bmatrix}{}
a_0 & a_{-1} & \cdots & a_{-n+2} & a_{-n+1} \\
a_1 & a_0 & a_{-1}   &  & a_{-n+2} \\
\vdots & a_1 & a_0 & \ddots & \vdots \\
a_{n-2} &  & \ddots & \ddots & a_{-1} \\
a_{n-1} & a_{n-2} &\cdots & a_1 & a_0
\end{bmatrix}. \]

We also let ${c}_n[f]$ be the $n$-by-$n$ optimal circulant preconditioner \cite{doi:10.1137/0909051} for ${A}_n[f]$, namely
 \[ c_{n}[f]=\begin{bmatrix}{}
c_0 & c_{n-1} & \cdots & c_{2} & c_{1} \\
c_1 & c_0 & c_{n-1}  &  & c_{2} \\
\vdots & c_1 & c_0 & \ddots & \vdots \\
c_{n-2} &  & \ddots & \ddots & c_{n-1} \\
c_{n-1} & c_{n-2} &\cdots & c_1 & c_0
\end{bmatrix} \] defined by 

\[ c_{k}= \begin{cases}
\frac{(n-k)a_k + k a_{k-n}}{n}, ~~0\leq k < n,\\
c_{n+k}, ~~-n < k < 0.\\
\end{cases}\]

\begin{lemma}\cite[Lemma 1 and 3]{doi:10.1137/0730062}
\label{lem:TCBoound}
If $f \in \mathcal{C}_{2\pi}$ we have

\[\|A_{n}[f]\|_2 \leq 2\|f\|_{\infty}~~\text{and}~~
\|c_{n}[f]\|_2 \leq 2\|f\|_{\infty}
\quad n=1,2,\dots.
\]\end{lemma}

Lemma \ref{lem:TCBoound} states that the $2$-norm of the circulant matrix and that of the Toeplitz matrix generated by $f$ are bounded by a constant which is independent of $n$.

\begin{lemma}\cite[Theorem~1]{doi:10.1137/0730062}
\label{lem:polyCT}
Let $f \in \mathcal{C}_{2\pi}$. Then for all $\epsilon>0$ there exists a positive integer $M>0$ such that for $n>2M$, we have
\[c_n[p_M] - A_n[p_M] = U_n - W_n,\] where  $p_M$ is a trigonometric polynomial such that $\|f-p_M\|_{\infty}<\epsilon$, $A_n[p_M]\in \mathbb{C}^{n \times n}$ is the Toeplitz matrix generated by $p_M$, $c_n[p_M]\in \mathbb{C}^{n \times n}$ is the optimal circulant preconditioner for $A_n[p_M]$, 
\[ \text{rank}~U_n \leq 2M\] and \[ \|W_n\|_2 \leq \frac{1}{n}M(M+1)(\epsilon+\|f\|_{\infty}). \]

\end{lemma}

Lemma \ref{lem:polyCT} indicates that the difference between the circulant matrix and the Toeplitz matrix generated by a trigonometric approximation to $f$ can be decomposed into the sum of a matrix of low rank and a matrix of small norm. In the next section, this lemma is used to prove that the difference between the matrix exponential of a circulant matrix and that of a Toeplitz matrix can also be decomposed in a similar fashion.

\section{Preliminaries on matrix functions}

In this section we introduce the preliminaries that will be used in the following section.

\begin{definition}\cite{MR2396439}
For any $A_n \in \mathbb{C}^{n\times n}$,

\[e^{A_n} = I_n+A_n+ \frac{1}{2!}A_n^2 + \frac{1}{3!}A_n^3 + \cdots,\]
\[\cos{A_n} = I_n-\frac{1}{2!}A_n^2 + \frac{1}{4!}A_n^4 - \frac{1}{6!}A_n^6 + \cdots\] and
\[\sin{A_n} = A_n-\frac{1}{3!}A_n^3 + \frac{1}{5!}A_n^5 - \frac{1}{7!}A_n^7 + \cdots.\]

\end{definition}

\begin{lemma}\cite[Theorem~10.2]{MR2396439}
\label{lemma:exp_plus_product}
For any $A_n, B_n\in \mathbb{C}^{n\times n}$
\[e^{(A_n+B_n)t} = e^{A_nt}e^{B_nt}\] for all $t$ if and only if $A_nB_n = B_nA_n$.
\end{lemma}

\begin{definition}\cite{MR2396439}
Given a vector norm on $\mathbb{C}^n$, the corresponding \emph{subordinate matrix norm} is defined by
\[\|A_n\| = \max_{x\neq0} \frac{\|A_n x\|}{\|x\|}.\]
\end{definition}

\begin{definition}\cite{MR2396439}
The norm $\|\cdot\|$ is called \emph{consistent} if 
\[\|AB\| \leq \|A\|\|B\|\] for all $ A \in \mathbb{C}^{m\times n}$ and $B\in \mathbb{C}^{n\times p}$.
\end{definition}

\begin{lemma}\cite[Problem~10.3]{MR2396439}
\label{lemma:expmbound}
For any subordinate matrix norm and any $A_n, B_n\in \mathbb{C}^{n\times n}$, we have
\[
\|e^{A_n} - e^{B_n} \| \leq \|A_n-B_n\| e^{\max{(\|A_n\|,\|B_n\|)}}.
\] 
\end{lemma}

\begin{lemma}\cite[Theorem~10.1]{MR2396439}
\label{lem:suzuki}
For $A_n \in \mathbb{C}^{n\times n}$, let \[ P_{r,s}=[\sum_{i=0}^{r} \frac{1}{i!}(\frac{A_n}{s})]^s.\]
Then for any consistent matrix norm \[ \|e^{A_n} - P_{r,s}\| \leq \frac{\|A_n\|^{r+1}}{s^r(r+1)!}e^{\|A_n\|}\] and $$\lim_{r\to\infty}P_{r,s}= \lim_{s\to\infty}P_{r,s}=e^{A_n}.$$
\end{lemma}

\begin{lemma}\cite[Theorem~10.10]{MR2396439}
\label{lem:expmAbd}
For $A_n \in \mathbb{C}^{n\times n}$ and any subordinate matrix norm,
\[
e^{-\|A_n\|}\leq\|e^{A_n}\| \leq e^{\|A_n\|}\quad n=1,2,\dots.
\]
\end{lemma}

\begin{lemma}\cite{doi:10.1137/140974213}
For any circulant matrix $C_n \in \mathbb{C}^{n\times n}$, the \emph{absolute value circulant matrix} $|C_n|$ is defined to be
\[
|C_n| = (C_n^* C_n)^{\frac{1}{2}}=(C_n C_n^*)^{\frac{1}{2}}=F_n^*|\Lambda_n|F_n,
\] where $|\Lambda_n|$ is the diagonal matrix in the eigenvalue decomposition of $C_n$ with all entries replaced by their magnitudes.
\end{lemma}

\section{Spectra of the preconditioned matrices }

In this section, for Hermitian $g(A_n[f])$ we show the preconditioned matrix $$g(c_n[f])^{-1}g(A_n[f])$$ can be decomposed into the sum of a unitary matrix, a matrix of low rank and a matrix of small norm for sufficiently large $n$ under some assumptions. For non-Hermitian $g(A_n[f])$, we consider its normal equations system and also show that the preconditioned matrix can also be decomposed in a similar way.

\subsection{Spectra of $(e^{c_n[f]})^{-1}e^{A_n[f]} $}

We first provide some lemmas concerning to the matrix exponential and then give the main results concerning to the preconditioned matrix $(e^{c_n[f]})^{-1}e^{A_n[f]} $.

\begin{corollary}
\label{coro:invexpmCBoound}
Let $f \in \mathcal{C}_{2\pi}$. Let $A_n[f]\in \mathbb{C}^{n \times n}$ be the Toeplitz matrix generated by $f$ and $c_n[f]\in \mathbb{C}^{n \times n}$ be the optimal circulant preconditioner for $A_n[f]$. We have
\[
\|(e^{c_{n}[f]})^{-1}\|_2 \leq e^{2\|f\|_{\infty}} \quad n=1,2,\dots.
\]
\end{corollary}
\begin{proof}
By Lemma \ref{lemma:exp_plus_product}, we have $$e^{c_{n}[f]} e^{-c_{n}[f]} =  e^{c_{n}[f]-c_{n}[f]} =I_n.$$ Thus $e^{-c_{n}[f]}$ is the inverse of $e^{c_{n}[f]}$. We then have $$\|(e^{c_{n}[f]})^{-1}\|_2 = \|e^{-c_{n}[f]}\|_2.$$ Using lemmas \ref{lem:TCBoound} and \ref{lem:expmAbd}, we have \[ \|e^{-c_{n}[f]}\|_2 \leq e^{\|c_n(f)\|_2} \leq e^{2\|f\|_{\infty}}. \]\qed
\end{proof}

We are now ready to give our main results on the spectrum of $(e^{c_n[f]})^{-1}e^{A_n[f]} $.

\begin{theorem}
\label{thm:aminCTdecomp_chap5}
Let $f \in \mathcal{C}_{2\pi}$. Let $A_n[f]\in \mathbb{C}^{n \times n}$ be the Toeplitz matrix generated by $f$ and $c_n[f]\in \mathbb{C}^{n \times n}$ be the optimal circulant preconditioner for $A_n[f]$. For all $\epsilon > 0$, there exist positive integers $N$ and $M$ such that for all $n > N$ \[ e^{c_n[f]} - e^{A_n[f]} = R_n[f] + E_n[f],\] where \[ \text{rank}~ R_n[f]\leq 2M, \] \[ \|E_n[f]\|_2 \leq \epsilon.\]
\end{theorem}

\begin{proof}
Since $f \in \mathcal{C}_{2\pi}$, by Weierstrass theorem \cite[Theorem~6.1]{MR604014}, for any $\epsilon>0$ there exists $M\in \mathbb{N}$ and a trigonometric polynomial \[ p_M(\theta) = \sum_{k=-M}^{k=M}\rho_k e^{\mathbf{i}k\theta}  \] such that \begin{equation} 
\label{eqn:Weierstrass}
 \|f-p_M\|_{\infty} \leq \epsilon.
\end{equation}
For all $n>2M$ we decompose \[ e^{c_n[f]}-e^{A_n[f]} = \underbrace{e^{c_n[f]}-e^{c_n[p_M]}}_{G_1} + \underbrace{e^{c_n[p_M]} - e^{A_n[p_M]}}_{B} + \underbrace{e^{A_n[p_M]}-e^{A_n[f]}}_{G_2}, \]
where  $A_n[p_M]\in \mathbb{C}^{n \times n}$ is the Toeplitz matrix generated by $p_M$ and $c_n[p_M]\in \mathbb{C}^{n \times n}$ is the optimal circulant preconditioner for $A_n[p_M]$.

We first want to show that $G_1 + G_2$ is of small norm. Using Lemmas \ref{lem:TCBoound}, \ref{lemma:expmbound} and (\ref{eqn:Weierstrass}) we have
\begin{eqnarray}\nonumber
\|G_1 + G_2\|_2&\leq&\|e^{c_n[f]}-e^{c_n[p_M]}\|_2 + \|e^{A_n[f]}-e^{A_n[p_M]}\|_2\\\nonumber
&\leq&\|c_n[f]-c_n[p_M]\|_2 e^{\max{(\|c_n[f]\|_2,\|c_n[p_M]\|_2})} \\\nonumber
&&~+ \|A_n[f]-A_n[p_M]\|_2 e^{\max{(\|A_n[f]\|_2,\|A_n[p_M]\|_2})}\\\nonumber
&\leq&( \|c_n[f]-c_n[p_M]\|_2  + \|A_n[f]-A_n[p_M]\|_2 )e^{2\max{(\|f\|_{\infty},\|p_M\|_{\infty}})} \\\nonumber
&\leq&( 2\|f-p_M]\|_{\infty}  + 2\|f-p_M]\|_{\infty} )e^{2\max{(\|f\|_{\infty},\|p_M\|_{\infty}})}\\
&\leq& (4 e^{2\max{(\|f\|_{\infty},\|p_M\|_{\infty}})}) \epsilon.
\label{eqn:G1G2}
\end{eqnarray}

We further rewrite 
\[ B = \underbrace{e^{c_n[p_M]} - \sum_{i=0}^{K}\frac{1}{i!}c_n[p_M]^i}_{B_1} + \underbrace{\sum_{i=0}^{K}\frac{1}{i!}c_n[p_M]^i - \sum_{i=0}^{K}\frac{1}{i!}A_n[p_M]^i}_{D} + \underbrace{\sum_{i=0}^{K}\frac{1}{i!}A_n[p_M]^i- e^{A_n[p_M]}}_{B_2}, \] where $K$ is a positive integer.

We are now to measure the norm of $B_1 + B_2$. Using Lemmas \ref{lem:TCBoound}, \ref{lem:suzuki} and (\ref{eqn:Weierstrass}), we have
\begin{eqnarray}\nonumber
\|B_1 + B_2\|_2&\leq&\|e^{c_n[p_M]} - \sum_{i=0}^{K}\frac{1}{i!}c_n[p_M]^i\|_2 + \|\sum_{i=0}^{K}\frac{1}{i!}A_n[p_M]^i- e^{A_n[p_M]}\|_2\\\nonumber
&\leq&\frac{\|c_n[p_M]\|_2^{K+1}}{(K+1)!}e^{\|c_n[p_M]\|_2} + \frac{\|A_n[p_M]\|_2^{K+1}}{(K+1)!}e^{\|A_n[p_M]\|_2}\\\nonumber
&\leq&(\frac{\|c_n[p_M]\|_2^{K+1}}{(K+1)!}+ \frac{\|A_n[p_M]\|_2^{K+1}}{(K+1)!})e^{2\|p_M\|_{\infty}}\\\nonumber
&\leq&\frac{(2\|p_M\|_{\infty})^{K+1}}{(K+1)!}2e^{2\|p_M\|_{\infty}}=:\epsilon_{K}
\end{eqnarray}
which converges to zero as $K$ goes to infinity. Therefore for a given $\epsilon_{K}>0$, there exists an integer $K$ such that
\begin{equation}
 \|B_1 + B_2\|_2\leq\epsilon_{K} \leq \epsilon. \label{eqn:B1B2}
 \end{equation}
We next show that $D$ can be decomposed into a sum of a matrix of low rank and a matrix of small norm. Firstly, we observe that \[  c_n[p_M] - A_n[p_M] = U_n - W_n,  \] where 
 \[ U_{n}=\begin{bmatrix}{}
& & & &\frac{n-M}{n}\rho_{M}&\cdots&\frac{n-1}{n}\rho_{1}\\
& & & & &\ddots&\vdots\\
& & & & &  &\frac{n-M}{n}\rho_{M}  \\
 & &  &  &   &  &  \\
\frac{n-M}{n}\rho_{-M} &  & &  &  & \\
\vdots & \ddots & &&  &   &  \\
\frac{n-1}{n}\rho_{-1} & \cdots& \frac{n-M}{n}\rho_{-M}  &&  &  & 
\end{bmatrix} \]
and
\[ W_{n}=\begin{bmatrix}{}
0&\frac{1}{n}\rho_{-1}&\cdots&\frac{M}{n}\rho_{-M}&&&&\\
\frac{1}{n}\rho_{1}&\ddots&\ddots&\ddots&\ddots&&&\\
\vdots&\ddots&\ddots&\ddots&\ddots&\ddots&&\\
\frac{M}{n}\rho_{M}&\ddots&\ddots&\ddots&\ddots&\ddots&\ddots&\\
&\ddots&\ddots&\ddots&\ddots&\ddots&\ddots&\frac{M}{n}\rho_{-M}\\
&&\ddots&\ddots&\ddots&\ddots&\ddots&\vdots\\
&&&\ddots&\ddots&\ddots&\ddots&\frac{1}{n}\rho_{-1}\\
&&&&\frac{M}{n}\rho_{M}&\cdots&\frac{1}{n}\rho_{1}   &0
\end{bmatrix}. 
\]

Rewrite $D$ as
\begin{eqnarray}\nonumber
D &=& \sum_{i=0}^{K}\frac{1}{i!}c_n[p_M]^i - \sum_{i=0}^{K}\frac{1}{i!}A_n[p_M]^i\\\nonumber
&=&\sum_{i=1}^{K}\frac{1}{i!}(c_n[p_M]^i - A_n[p_M]^i)\\\nonumber
&=&\sum_{i=1}^{K}\frac{1}{i!}( \sum_{j=0}^{i-1}c_n[p_M]^{j}(c_n[p_M]-A_n[p_M])A_n[p_M]^{i-1-j} )\\\nonumber
&=&\sum_{i=1}^{K}\frac{1}{i!}( \sum_{j=0}^{i-1}c_n[p_M]^{j}(U_n - W_n)A_n[p_M]^{i-1-j} )\\\nonumber
&=&\underbrace{ \sum_{i=1}^{K}\frac{1}{i!}( \sum_{j=0}^{i-1}c_n[p_M]^{j}U_nA_n[p_M]^{i-1-j} )}_{R_n[f]} + \underbrace{\sum_{i=1}^{K}\frac{1}{i!}( \sum_{j=0}^{i-1}c_n[p_M]^{j}W_nA_n[p_M]^{i-1-j} )}_{J}.\\\nonumber
 \end{eqnarray}

Using Lemmas \ref{lem:TCBoound} and \ref{lem:polyCT}, we can estimate the norm of $J$:
\begin{eqnarray}\nonumber
\|J\|_2 &=& \|\sum_{i=1}^{K}\frac{1}{i!} \sum_{j=0}^{i-1}c_n[p_M]^{j}W_nA_n[p_M]^{i-1-j} \|_2\\\nonumber
&\leq &\sum_{i=1}^{K}\frac{1}{i!}\| \sum_{j=0}^{i-1}c_n[p_M]^{j}W_nA_n[p_M]^{i-1-j} \|_2\\\nonumber
&\leq &\|W_n\|_2\sum_{i=1}^{K}\frac{1}{i!} \sum_{j=0}^{i-1}\|c_n[p_M]\|_2^{j}\|A_n[p_M]\|_2^{i-1-j}\\\nonumber
&\leq &\|W_n\|_2\sum_{i=1}^{K}\frac{1}{i!} \sum_{j=0}^{i-1}(2\|p_M\|_{\infty})^{j}(2\|p_M\|_{\infty})^{i-1-j}\\\nonumber
&= &\|W_n\|_2\sum_{i=1}^{K}\frac{1}{i!} \sum_{j=0}^{i-1}(2\|p_M\|_{\infty})^{i-1}\\\nonumber
&= &\|W_n\|_2\sum_{i=1}^{K}\frac{1}{(i-1)!} (2\|p_M\|_{\infty})^{i-1}\\\nonumber
&\leq &\|W_n\|_2\sum_{i=1}^{\infty}\frac{1}{(i-1)!} (2\|p_M\|_{\infty})^{i-1}\\\nonumber
&= &\|W_n\|_2 e^{2\|p_M\|_{\infty}}\\\label{eqn:Jnorm}
&\leq &\frac{1}{n}M(M+1)(\epsilon +\|f\|_{\infty}) e^{2\|p_M\|_{\infty}}.\\\nonumber
\end{eqnarray}

We now show that rank $R_n[f] \leq 2KM$ by first investigating the structure of $R_n[f]$. Similar to the approach used in the proof of Lemma 3.11 in \cite{doi:10.1137/080720280}, simple computations show that
 \[ 
 c_n[p_M]^{\alpha}U_nA_n[p_M]^{\beta}=\begin{bmatrix}{}
\Diamond & \cdots & \Diamond & & & \Diamond & \cdots& \Diamond \\
\vdots  & \Diamond  & \vdots  &  &&\vdots  & \Diamond& \vdots \\
\Diamond & \cdots  & \Diamond& & & \Diamond & \cdots& \Diamond \\
  &   & &  &  && &  \\
  &   & &  &  && &  \\
\Diamond & \cdots & \Diamond & & & \Diamond & \cdots& \Diamond \\
\vdots  & \Diamond  & \vdots  & & &\vdots  &\Diamond & \vdots\\
\Diamond & \cdots  & \Diamond& & & \Diamond & \cdots& \Diamond 
\end{bmatrix},
  \]
where the diamonds represent the non-zero entries which appear only in the four $(\alpha+1)M$ by $(\beta+1)M$ blocks in the corners, provided that $n$ is larger than $2\max(\alpha+1, \beta+1)M$. Since the rank of \[R_n[f]=\sum_{i=1}^{K}\frac{1}{i!}( \sum_{j=0}^{i-1}c_n[p_M]^{j}U_nA_n[p_M]^{i-1-j} )\] is determined by that of $\sum_{j=0}^{K-1}c_n[p_M]^{j}U_nA_n[p_M]^{K-1-j}$ which is a block matrix with only four non-zero $KM$ by $KM$ blocks in its corners, it follows that the rank of $R_n[f]$ is less than or equal to $2KM$ if we assume $n>2KM$.

Considering (\ref{eqn:Jnorm}), we pick \[N := \max{\{M(M+1)(1+\frac{\|f\|_{\infty}}{\epsilon})e^{2\|p_M\|_{\infty}}, 2KM  \}}, \] and it follows that for all $n>N$ we have $\|J\|_2\leq \epsilon$. Further combining this result with (\ref{eqn:G1G2}) and (\ref{eqn:B1B2}), we conclude that for all $n>N$ \[ \|E_n[f]\|_2=\|G_1+B_1+J+B_2+G_2\|_2 \leq (4 e^{2\max{(\|f\|_{\infty},\|p_M\|_{\infty}})}+2) \epsilon. \]\qed
\end{proof}

\begin{corollary}
\label{coro:maininvexpmCT}
Let $f \in \mathcal{C}_{2\pi}$. Let $A_n[f]\in \mathbb{C}^{n \times n}$ be the Toeplitz matrix generated by $f$ and $c_n[f]\in \mathbb{C}^{n \times n}$ be the optimal circulant preconditioner for $A_n[f]$. For all $\epsilon > 0$, there exist positive integers $N$ and $M$ such that for all $n > N$ \[ (e^{c_n[f]})^{-1}e^{A_n[f]}=I_n + \widehat{R}_n[f] + \widehat{E}_n[f],\] where \[ \text{rank}~\widehat{R}_n[f]\leq 2M, \] \[ \|\widehat{E}_n[f]\|_2 \leq \epsilon.\] 

\end{corollary}
\begin{proof}

By Theorem \ref{thm:aminCTdecomp_chap5}, we know that for all $\epsilon > 0$, there exist positive integers $N$ and $M$ such that for all $n > N$ \[ e^{c_n[f]} - e^{A_n[f]} = R_n[f] + E_n[f],\] where \[ \text{rank}~ R_n[f]\leq 2M, \] \[ \|E_n[f]\|_2 \leq \epsilon.\]

By Corollary \ref{coro:invexpmCBoound} we know that $\|(e^{c_n[f]})^{-1}\|_2$ is uniformly bounded with respect to $n$, so that we have 
\begin{eqnarray}\nonumber
(e^{c_n[f]})^{-1}e^{A_n[f]} &=&I_n+ (e^{c_n[f]})^{-1}(e^{A_n[f]}-e^{c_n[f]}) \\\nonumber
&=&I_n+ \underbrace{(e^{c_n[f]})^{-1}(-R_n[f])}_{\widehat{R_n}[f]}+\underbrace{(e^{c_n[f]})^{-1}(-E_n[f])}_{\widehat{E_n}[f]}.
\end{eqnarray}The result follows.
\qed
\end{proof}

\begin{remark}
Since $A_n[f]$ is Hermitian when $f$ is real-valued, we can write $A_n[f]=Z_n^{T}D_nZ_n$ where $D_n$ is a diagonal matrix with real eigenvalues $d_i$ being the eigenvalues of $A_n[f]$. We immediately see that  $e^{A_n[f]}=Z_n^{T} e^{D_n}Z_n$ is positive definite as its eigenvalues are all of the form $e^{d_i}>0$. Thus CG can be used in this case. \end{remark}

Next, for the more general case when $e^{A_n[f]}$ is non-Hermitian, we can use CG for the normal equations system with the preconditioner  $e^{c_n[f]}$.
\begin{corollary}
\label{coro:maininvexpmNORMALMETRIC}
Let $f \in \mathcal{C}_{2\pi}$. Let $A_n[f]\in \mathbb{C}^{n \times n}$ be the Toeplitz matrix generated by $f$ and $c_n[f]\in \mathbb{C}^{n \times n}$ be the optimal circulant preconditioner for $A_n[f]$. For all $\epsilon > 0$, there exist positive integers $N$ and $M$ such that for all $n > N$ \[ [(e^{c_n[f]})^{-1} e^{A_n[f]}]^{*}[(e^{c_n[f]})^{-1} e^{A_n[f]}]=I_n+ \overline{R}_n[f]+\overline{E}_n[f],\] where \[ \text{rank}~\overline{R}_n[f]\leq 4M, \] \[ \|\overline{E}_n[f]\|_2 \leq \epsilon.\]
\end{corollary}

\begin{proof}
By Corollary \ref{coro:maininvexpmCT}, we know that for all $\epsilon > 0$ there exist positive integers $N$ and $M$ such that for all $n > N$ \[ (e^{c_n[f]})^{-1}e^{A_n[f]}=I_n + \widehat{R}_n[f] + \widehat{E}_n[f],\] where \[ \text{rank}~\widehat{R}_n[f]\leq 2M, \] \[ \|\widehat{E}_n[f]\|_2 \leq \epsilon.\]

We then have
\begin{eqnarray}\nonumber
&&[(e^{c_n[f]})^{-1} e^{A_n[f]}]^{*}[(e^{c_n[f]})^{-1} e^{A_n[f]}]\\\nonumber
 &=& (I_n + \widehat{R_n}[f] + \widehat{E_n}[f])^{*} (I_n + \widehat{R}_n[f] + \widehat{E}_n[f])\\\nonumber
&=& I_n + \underbrace{\widehat{R}_n[f]^{*}(I_n+\widehat{R}_n[f]+\widehat{E}_n[f]) + (I_n+\widehat{E}_n[f]^{*})\widehat{R}_n[f]}_{\overline{R}_n[f]} \\\nonumber
&& +~\underbrace{\widehat{E}_n[f]+\widehat{E}_n[f]^{*}+\widehat{E}_n[f]^{*}\widehat{E}_n[f]}_{\overline{E}_n[f]}.
\end{eqnarray}

It immediately follows that rank $\overline{R}_n[f]\leq 4M$ and $\|\overline{E}_n[f]\|_2\leq \epsilon^2 + 2\epsilon$. 
\qed
\end{proof}

Since $[(e^{c_n[f]})^{-1} e^{A_n[f]}]^{*}[(e^{c_n[f]})^{-1} e^{A_n[f]}]$ in Corollary \ref{coro:maininvexpmNORMALMETRIC} is Hermitian, by Weyl's theorem we know that its eigenvalues are mostly close to $1$ when $n$ is sufficiently large.

\subsection{Spectra of $(\sin{c_n[f]})^{-1}\sin{A_n[f]} $ and $(\cos{c_n[f]})^{-1}\cos{A_n[f]} $}

In this subsection we directly show that similar results hold for $(\sin{c_n[f]})^{-1}\sin{A_n[f]} $ and $(\cos{c_n[f]})^{-1}\cos{A_n[f]} $ using the theorems on $(e^{c_n[f]})^{-1}e^{A_n[f]}$.

\begin{theorem}
\label{thm:aminSINETdecomp}
Let $f \in \mathcal{C}_{2\pi}$. Let $A_n[f]\in \mathbb{C}^{n \times n}$ be the Toeplitz matrix generated by $f$ and $c_n[f]\in \mathbb{C}^{n \times n}$ be the optimal circulant preconditioner for $A_n[f]$. For all $\epsilon > 0$, there exist positive integers $N$ and $M$ such that for all $n > N$ \[ \sin{{c_n[f]}} - \sin{A_n[f]} = \mathcal{R}_n[f] + \mathcal{E}_n[f],\] where \[ \text{rank } \mathcal{R}_n[f]\leq 2M, \] \[ \|\mathcal{E}_n[f]\|_2 \leq \epsilon.\]
\end{theorem}

\begin{proof}
By Theorem \ref{thm:aminCTdecomp_chap5}, we know that for all $\epsilon > 0$ there exist positive integers $N$ and $M$ such that for all $n > N$ \[ e^{c_n[f]} - e^{A_n[f]} = R_n[f] + E_n[f],\] where \[ \text{rank}~ R_n[f]\leq 2M, \] \[ \|E_n[f]\|_2 \leq \epsilon.\]

Using the fact that $\sin{A_n}=\frac{e^{\mathbf{i}A_n}-e^{-\mathbf{i}A_n}}{2\mathbf{i}}$ for any $A_n \in \mathbb{C}^{n\times n}$, we write
\begin{eqnarray}\nonumber
\sin{{c_n[f]}} - \sin{A_n[f]} &=& (\frac{e^{\mathbf{i}c_n[f]}-e^{-\mathbf{i}c_n[f]}}{2\mathbf{i}}) -  (\frac{e^{\mathbf{i}A_n[f]}-e^{-\mathbf{i}A_n[f]}}{2\mathbf{i}}) \\\nonumber
&=& (\frac{e^{\mathbf{i}c_n[f]}-e^{\mathbf{i}A_n[f]}}{2\mathbf{i}}) -  (\frac{e^{-\mathbf{i}c_n[f]}-e^{-\mathbf{i}A_n[f]}}{2\mathbf{i}}) \\\nonumber
&=& (\frac{e^{c_n[\mathbf{i}f]}-e^{A_n[\mathbf{i}f]}}{2\mathbf{i}}) -  (\frac{e^{c_n[-\mathbf{i}f]}-e^{A_n[-\mathbf{i}f]}}{2\mathbf{i}}) \\\nonumber
&=& (\frac{R_n[\mathbf{i}f] + E_n[\mathbf{i}f]}{2\mathbf{i}}) -  (\frac{R_n[-\mathbf{i}f] + E_n[-\mathbf{i}f]}{2\mathbf{i}}) \\\nonumber
&=& \underbrace{(\frac{R_n[\mathbf{i}f] - R_n[-\mathbf{i}f]}{2\mathbf{i}})}_{\mathcal{R}_n[f]} +  \underbrace{(\frac{E_n[\mathbf{i}f] - E_n[-\mathbf{i}f]}{2\mathbf{i}})}_{\mathcal{E}_n[f]}. \\\nonumber
\end{eqnarray}

Since $R_n[\mathbf{i}f]$ and $R_n[-\mathbf{i}f]$ are both block matrices with only four non-zero $M$ by $M$ blocks in their corners, we see that the rank of $\mathcal{R}_n[f]$ is less than or equal to $2M$. Also
\begin{eqnarray}\nonumber
\|\mathcal{E}_n[f]\|_2 &=& \|\frac{E_n[\mathbf{i} f] - E_n[-\mathbf{i} f]}{2\mathbf{i}}\|_2\\\nonumber
&\leq& \frac{\|E_n[\mathbf{i}f]\|_2   +  \|E_n[-\mathbf{i}f]\|_2}{2}\\\nonumber
&\leq& \epsilon.
\end{eqnarray}
\qed
\end{proof}

\begin{corollary}
\label{coro:aminSINETprecon}
Let $f \in \mathcal{C}_{2\pi}$. Let $A_n[f]\in \mathbb{C}^{n \times n}$ be the Toeplitz matrix generated by $f$ and $c_n[f]\in \mathbb{C}^{n \times n}$ be the optimal circulant preconditioner for $A_n[f]$. If $\| (\sin{c_n[f]})^{-1}  \|_2$ is uniformly bounded with respect to $n$, then for all $\epsilon > 0$ there exist positive integers $N$ and $M$ such that for all $n > N$ \[ (\sin{c_n[f]})^{-1}(\sin{A_n[f]})=I_n + \widehat{\mathcal{R}}_n[f] + \widehat{\mathcal{E}}_n[f],\] where \[ \text{rank } \widehat{\mathcal{R}}_n[f]\leq 2M, \] \[ \|\widehat{\mathcal{E}}_n[f]\|_2 \leq \epsilon.\]
\end{corollary}

\begin{proof}
By Theorem \ref{thm:aminSINETdecomp}, we know that for all $\epsilon > 0$, there exist positive integers $N$ and $M$ such that for all $n > N$ \[ \sin{{c_n[f]}} - \sin{A_n[f]} = \mathcal{R}_n[f] + \mathcal{E}_n[f],\] where \[ \text{rank } \mathcal{R}_n[f]\leq 2M, \] \[ \|\mathcal{E}_n[f]\|_2 \leq \epsilon.\]

By the assumption that $\| (\sin{c_n[f]})^{-1}  \|_2$ is uniformly bounded with respect to $n$, we have 
\begin{eqnarray}
(\sin{c_n[f]})^{-1}\sin{A_n[f]} &=&I_n+ (\sin{c_n[f]})^{-1}\nonumber
(\sin{A_n[f]}-\sin{c_n[f]})\\ \nonumber
 &=& I_n+\underbrace{(\sin{c_n[f]})^{-1}(-\mathcal{R}_n[f])}_{\widehat{\mathcal{R}}_n[f]}+\underbrace{(\sin{c_n[f]})^{-1}(-\mathcal{E}_n[f])}_{\widehat{\mathcal{E}}_n[f]}.
 \end{eqnarray} The result follows. \qed
\end{proof}

\begin{remark}
From \[ \|(\sin{c_n[f]})^{-1}\|_2=\max_i{|\frac{1}{\sin{\lambda_i}} |} \] we know that $\|(\sin{c_n[f]})^{-1}\|_2$ could be arbitrarily large since $\sin{\lambda_i}$ could be close to zero, where $\lambda_i$ is the $i$-th eigenvalue of $c_n[f]$. Therefore, we have needed to assume that $\|(\sin{c_n[f]})^{-1}\|_2$ is uniformly bounded with respect to $n$.
\end{remark}

Consider now the special case when $\sin{A_n[f]}$ is Hermitian. Unlike the case with the matrix exponential, we cannot use CG for $\sin{A_n[f]}$ as it is not positive definite in general. By the diagonalisation of  $\sin{A_n[f]}=Z_n^{T} (\sin{D_n})Z_n$ where $D_n$ is a diagonal matrix with real eigenvalues $d_i$ being the eigenvalues of $A_n[f]$, as before, we see that its eigenvalues are all of the form $-1\leq\sin{d_i}\leq1$. Krylov subspace methods like MINRES \cite{paige1975solution} together with a Hermitian positive definite preconditioner $|\sin{c_n[f]}|$ should be used \cite[Section~5]{ANU:9672992}, where $|\sin{c_n[f]}|$ is the absolute value circulant preconditioner \cite{doi:10.1137/140974213,McDonald2017} of $\sin{c_n[f]}$.

\begin{corollary}
\label{coro:maininvSINSYMMETRIC}
Let $f \in \mathcal{C}_{2\pi}$ be a real-valued function. Let $A_n[f]\in \mathbb{C}^{n \times n}$ be the Toeplitz matrix generated by $f$ and $c_n[f]\in \mathbb{C}^{n \times n}$ be the optimal circulant preconditioner for $A_n[f]$. If $\| (\sin{c_n[f]})^{-1}  \|_2$ is uniformly bounded with respect to $n$, then for all $\epsilon > 0$ there exist positive integers $N$ and $M$ such that for all $n > N$  \[ |\sin{c_n[f]}|^{-1}\sin{A_n[f]}=Q_n + \widetilde{\mathcal{R}}_n[f] + \widetilde{\mathcal{E}}_n[f], \] where $Q_n$ is unitary and Hermitian, \[ \text{rank } \widetilde{\mathcal{R}}_n[f]\leq 2M, \] \[ \|\widetilde{\mathcal{E}}_n[f]\|_2 \leq \epsilon.\] 
\end{corollary}

\begin{proof}
Using a similar approach proposed in \cite{doi:10.1137/140974213}, we want to show that $|\sin{c_n[f]}|$ is an effective Hermitian positive definite preconditioner for $\sin{A_n[f]}$ under the assumptions.

As $\sin{c_n[f]}$ is a circulant matrix we write $\sin{c_n[f]} = F_n^* (\sin{\Lambda_n})F_n$ where $\sin{\Lambda_n}$ is the diagonal matrix with the eigenvalues of $\sin{c_n[f]}$. We then immediately have 
\begin{eqnarray}\nonumber
|\sin{c_n[f]}| &=& F_n^* |\sin{\Lambda_n}|F_n \\\nonumber
&=&  F_n^* (\sin{\Lambda_n})F_n \underbrace{F_n^* (\text{sign}(\sin{\Lambda_n}))^{-1}F_n}_{Q_n}\\\label{eqn:abdC}
&=& \sin{c_n[f]} Q_n,
\end{eqnarray}
where $\text{sign}(\sin{\Lambda_n})= \text{diag}(\frac{\sin{\lambda_i}}{|\sin{\lambda_i|}})=\text{diag}(\pm1)$ and $Q_n$ is both unitary and involutory (i.e. $Q_n^2=I_n$). It is noted that $|\sin{\lambda_i|} \neq 0$ for $i=1,2, \dots,n$ by the assumption, so that $\text{sign}(\sin{\Lambda_n})$ is well defined.

By Corollary \ref{coro:aminSINETprecon}, we know that for all $\epsilon > 0$ there exist positive integers $N$ and $M$ such that for all $n > N$ \[ (\sin{c_n[f]})^{-1}(\sin{A_n[f]})=I_n + \widehat{\mathcal{R}}_n[f] + \widehat{\mathcal{E}}_n[f],\] where \[ \text{rank } \widehat{\mathcal{R}}_n[f]\leq 2M, \] \[ \|\widehat{\mathcal{E}}_n[f]\|_2 \leq \epsilon.\]

Using (\ref{eqn:abdC}), we have \[ |\sin{c_n[f]}|^{-1}\sin{A_n[f]} = Q_n \sin{c_n[f]}^{-1}\sin{A_n[f]} = Q_n + \underbrace{Q_n\widehat{\mathcal{R}}_n[f]}_{\widetilde{\mathcal{R}}_n[f]} + \underbrace{Q_n\widehat{\mathcal{E}}_n[f]}_{\widetilde{\mathcal{E}}_n[f]}.  \]

Since $Q_n$ is unitary, we know 
\[
\text{rank}~\widetilde{\mathcal{R}}_n[f] =\text{rank}(Q_n\widehat{\mathcal{R}}_n[f])= \text{rank}~\widehat{\mathcal{R}}_n[f] \leq 2M
\] and 
\[
\|\widetilde{\mathcal{E}}_n[f]\|_2 = \|Q_n\widehat{\mathcal{E}}_n[f]\|_2 =  \|\widehat{\mathcal{E}}_n[f]\|_2 \leq \epsilon.
\] The result follows.
\qed
\end{proof}

\begin{remark}
Since $|\sin{c_n[f]}|$ is also a circulant matrix, $|\sin{c_n[f]}|^{-1}\mathbf{d}$ for any vector $\mathbf{d}$ can be efficiently computed by several FFTs in $\mathcal{O}(n\log{n})$ operations.
\end{remark} 

For the more general case when $\sin{A_n[f]}$ is non-Hermitian, we can use CG for the normal equations system with the preconditioner  $\sin{c_n[f]}$.

\begin{corollary}
\label{coro:maininvsinNORMALMETRIC}
Let $f \in \mathcal{C}_{2\pi}$. Let $A_n[f]\in \mathbb{C}^{n \times n}$ be the Toeplitz matrix generated by $f$ and $c_n[f]\in \mathbb{C}^{n \times n}$ be the optimal circulant preconditioner for $A_n[f]$. If $\|(\sin{c_n[f]})^{-1}  \|_2$ is uniformly bounded with respect to $n$, then for all $\epsilon > 0$ there exist positive integers $N$ and $M$ such that for all $n > N$ 
\[ [(\sin{c_n[f]})^{-1} \sin{A_n[f]}]^{*}[(\sin{c_n[f]})^{-1} \sin{A_n[f]}]=I_n+ \overline{\mathcal{R}}_n[f]+\overline{\mathcal{E}}_n[f],\] where \[ \text{rank}~\overline{\mathcal{R}}_n[f]\leq 4M, \] \[ \|\overline{\mathcal{E}}_n[f]\|_2 \leq \epsilon.\]

\end{corollary}

\begin{proof}
By Corollary \ref{coro:aminSINETprecon}, we know that for all $\epsilon > 0$ there exist positive integers $N$ and $M$ such that for all $n > N$ \[ (\sin{c_n[f]})^{-1}\sin{A_n[f]}=I_n + \widehat{\mathcal{R}_n}[f] + \widehat{\mathcal{E}_n}[f],\] where \[ \text{rank}~\widehat{\mathcal{R}}_n[f]\leq 2M, \] \[ \|\widehat{\mathcal{E}}_n[f]\|_2 \leq \epsilon.\]

We then have
\begin{eqnarray}\nonumber
&&[(\sin{c_n[f]})^{-1} \sin{A_n[f]}]^{*}[(\sin{c_n[f]})^{-1} \sin{A_n[f]}]\\\nonumber
 &=& (I_n + \widehat{\mathcal{R}}_n[f] + \widehat{\mathcal{E}}_n[f])^{*} (I_n + \widehat{\mathcal{R}}_n[f] + \widehat{\mathcal{E}}_n[f])\\\nonumber
&=& I_n + \underbrace{\widehat{\mathcal{R}}_n[f]^{*}(I_n+\widehat{\mathcal{R}}_n[f]+\widehat{\mathcal{E}}_n[f]) + (I_n+\widehat{\mathcal{E}_n}[f]^{*})\widehat{\mathcal{R}}_n[f]}_{\overline{\mathcal{R}}_n[f]} \\\nonumber
&& +~\underbrace{\widehat{\mathcal{E}}_n[f]+\widehat{\mathcal{E}}_n[f]^{*}+\widehat{\mathcal{E}}_n[f]^{*}\widehat{\mathcal{E}}_n[f]}_{\overline{\mathcal{E}}_n[f]}.
\end{eqnarray}

It immediately follows that rank $\overline{\mathcal{R}}_n[f]\leq 4M$ and $\|\overline{\mathcal{E}}_n[f]\|_2\leq\epsilon^2 + 2\epsilon$. 
\qed
\end{proof}

Since $[(\sin{c_n[f]})^{-1} \sin{A_n[f]}]^{*}[(\sin{c_n[f]})^{-1} \sin{A_n[f]}]$ in Corollary \ref{coro:maininvsinNORMALMETRIC} is Hermitian, by Weyl's theorem, we know that its eigenvalues are mostly close to $1$ when $n$ is sufficiently large.

Because $\cos{A_n}=\frac{e^{\mathbf{i}A_n}+e^{-\mathbf{i}A_n}}{2}$ for any $A_n \in \mathbb{C}^{n\times n}$, we have the following similar theorem and corollaries for $\cos{A_n[f]}$.

\begin{theorem}
\label{thm:aminCOSETdecomp}
Let $f \in \mathcal{C}_{2\pi}$. Let $A_n[f]\in \mathbb{C}^{n \times n}$ be the Toeplitz matrix generated by $f$ and $c_n[f]\in \mathbb{C}^{n \times n}$ be the optimal circulant preconditioner for $A_n[f]$. For all $\epsilon > 0$, there exist positive integers $N$ and $M$ such that for all $n > N$ \[ \cos{{c_n[f]}} - \cos{A_n[f]} = \mathscr{R}_n[f] + \mathscr{E}_n[f],\] where \[ \text{rank } \mathscr{R}_n[f]\leq 2M, \] \[ \|\mathscr{E}_n[f]\|_2 \leq \epsilon.\]
\end{theorem}

\begin{corollary}
\label{coro:aminCOSETprecon}
Let $f \in \mathcal{C}_{2\pi}$. Let $A_n[f]\in \mathbb{C}^{n \times n}$ be the Toeplitz matrix generated by $f$ and $c_n[f]\in \mathbb{C}^{n \times n}$ be the optimal circulant preconditioner for $A_n[f]$. If $\| (\cos{c_n[f]})^{-1}  \|_2$ is uniformly bounded with respect to $n$, then for all $\epsilon > 0$ there exist positive integers $N$ and $M$ such that for all $n > N$  \[ |\cos{c_n[f]}|^{-1}\cos{A_n[f]}=Q_n + \widetilde{\mathscr{R}}_n[f] + \widetilde{\mathscr{E}}_n[f], \] where $Q_n$ is unitary and Hermitian, \[ \text{rank } \widetilde{\mathscr{R}}_n[f]\leq 2M, \] \[ \|\widetilde{\mathscr{E}}_n[f]\|_2 \leq \epsilon.\] 
\end{corollary}

For the more general case when $\cos{A_n[f]}$ is non-Hermitian, we can use CG for the normal equations system with the preconditioner  $\cos{c_n[f]}$.

\begin{corollary}
\label{coro:maininvcosNORMALMETRIC}
Let $f \in \mathcal{C}_{2\pi}$. Let $A_n[f]\in \mathbb{C}^{n \times n}$ be the Toeplitz matrix generated by $f$ and $c_n[f]\in \mathbb{C}^{n \times n}$ be the optimal circulant preconditioner for $A_n[f]$. If $\| (\cos{c_n[f]})^{-1}  \|_2$ is uniformly bounded with respect to $n$, then for all $\epsilon > 0$ there exist positive integers $N$ and $M$ such that for all $n > N$ 
\[ [(\cos{c_n[f]})^{-1} \cos{A_n[f]}]^{*}[(\cos{c_n[f]})^{-1} \cos{A_n[f]}]=I_n+ \overline{\mathscr{R}}_n[f]+\overline{\mathscr{E}}_n[f],\] where \[ \text{rank}~\overline{\mathscr{R}}_n[f]\leq 4M, \] \[ \|\overline{\mathscr{E}}_n[f]\|_2 \leq \epsilon.\]

\end{corollary}

Since $[(\cos{c_n[f]})^{-1} \cos{A_n[f]}]^{*}[(\cos{c_n[f]})^{-1} \cos{A_n[f]}] $ in Corollary \ref{coro:maininvcosNORMALMETRIC} is Hermitian, by Weyl's theorem, we know that its eigenvalues are mostly close to $1$ when $n$ is sufficiently large.

\section{Extension to analytic functions of Toeplitz matrices}

\begin{lemma}\cite[Theorem~1.18]{MR2396439}
\label{lem:analyticfunctionproperties}
Let $h$ be analytic on an open subset $\Omega\subseteq \mathbb{C}$ such that each connected component of $\Omega$ is closed under conjugation. Consider the corresponding matrix function $h$ on its natural domain in $\mathbb{C}^{n \times n}$, the set $\mathcal{D}=\{A_n \in \mathbb{C}^{n \times n}: \Lambda(A_n)\subseteq \Omega\}$. Then  the following are equivalent:

{(a)} $h(A_n^*)=h(A_n)^*$ for all $A_n \in \mathcal{D}$.
\vspace{1mm}

{(b)} $h(\overline{A_n})=\overline{h(A_n)}$ for all $A_n \in \mathcal{D}$.
\vspace{1mm}

{(c)} $h(\mathbb{R}^{n\times n}\cap\mathcal{D})\subseteq \mathbb{R}^{n\times n}$.
\vspace{1mm}

{(d)} $h(\mathbb{R}\cap\Omega)\subseteq \mathbb{R}$.

\end{lemma}

\begin{lemma}\cite[Theorem~4.7]{MR2396439}
\label{lem:matrixtaylorseries}
Suppose $h$ has a Taylor series expansion \[ h(z)=\sum_{k=0}^{\infty}a_k(z-\alpha)^k, \] where $a_k=\frac{h^{(k)}(\alpha)}{k!}$, with radius of convergence $r$. If $A_n\in \mathbb{C}^{n\times n}$ then $f(A_n)$ is defined and is given by \[h(A_n) = \sum_{k=0}^{\infty}a_k(A_n-\alpha I_n)^{k}\] if and only if the distinct eigenvalues $\lambda_1$, $\cdots$, $\lambda_s$ of $A_n$ satisfy one of the conditions

{(a)} $|\lambda_i-\alpha|<r$,
\vspace{1mm}

{(b)} $|\lambda_i-\alpha|=r$ and the series for $h^{(n_i-1)}(\lambda)$, where $n_i$ is the index of $\lambda_i$, is convergent at the point $\lambda=\lambda_i$, $i=1,\dots,s$.
\end{lemma}

\begin{lemma}\cite[Theorem~4.8]{MR2396439}
\label{lem:matrixtaylorserieserrorbd}
Suppose $h$ has a Taylor series expansion \[ h(z)=\sum_{k=0}^{\infty}a_k(z-\alpha)^k, \] where $a_k=\frac{h^{(k)}(\alpha)}{k!}$, with radius of convergence $r$. If $A_n\in \mathbb{C}^{n\times n}$ with $\rho(A_n-\alpha I_n)<r$ then for any matrix norm $\|\cdot\|$ \[\|h(A_n) - \sum_{k=0}^{K-1}a_k(A_n-\alpha I_n)^{k}\| \leq \frac{1}{K!}\max_{0\leq t \leq 1} {\|(A_n-\alpha I_n)^K h^{(K)}(\alpha I_n + t (A_n-\alpha I_n)) \|}. \]
\end{lemma}

We first assume that the condition in Lemma \ref{lem:matrixtaylorseries} is satisfied so $h(A_n)$ can be represented by a Taylor series of $A_n$. Replacing Lemma \ref{lem:suzuki} with Lemma \ref{lem:matrixtaylorserieserrorbd}, we can show that $h(c_n[f])-h(A_n[f])$ can be decomposed into a sum of a matrix of certain rank and a small norm matrix in a similar manner to theorem \ref{thm:aminCTdecomp_chap5}. Further assuming the boundedness of $h(c_n[f])$, we can prove a similar decomposition for $h(c_n[f])^{-1}h(A_n[f])$ or its normal equations matrix.

\section{Numerical results}

In this section, we demonstrate the effectiveness of our proposed preconditioners $g(c_n[f])$ for the systems $g(A_n[f])\mathbf{x}=\mathbf{b}$ using CG, MINRES and GMRES \cite{MR848568}. Throughout all numerical tests, $e^{A_n[f]}$ is computed by the MATLAB built-in function \textbf{expm} whilst $\sin{A_n[f]}$ and $\cos{A_n[f]}$ are computed by \textbf{funm}. Also, we use the function \textbf{pcg} to solve the Hermitian positive definite systems \[ g(A_n[f])\mathbf{x}=\mathbf{b}, \] and
 \[ g(A_n[f])^{*} g(A_n[f])\mathbf{x}= g(A_n[f])^{*}\mathbf{b},\]
 where $\mathbf{b}$ is generated by the function \textbf{randn(n,1)}, with the zero vector as the initial guess. For Hermitian indefinite systems, we use the function \textbf{minres}. As a comparison, GMRES is also used for some systems and it is executed by \textbf{gmres}. The stopping criterion used is \[ \frac{\|\mathbf{r}_j\|_2}{\|\mathbf{b}\|_2} < 10^{-7}, \] where $\mathbf{r}_j$ is the residual vector after $j$ iterations.

Example 1: We first consider $e^{A_n[f]}$, where $A_n[f]$ is generated by several functions $f$ with moderate $\|f\|_{\infty}$. Table \ref{tab:ExpTCG_ch2_hon16} shows the numbers of iterations needed for $e^{A_n[f]}$ with $A_n[f]$ generated by ${f(\theta) = \frac{4}{3}\theta\cos(\theta)}$ with or without preconditioners. It is clear that the proposed precondtioner is efficient for speeding up the rate of convergence of CG. In Figure \ref{fig:exp_512_CG_ch2_hon16} {(a)} and {(b)}, the contrast between the spectra of the matrices is shown. In Figure \ref{fig:exp_512_CG_ch2_hon16} {(c)}, we see that the eigenvalues of the preconditioned matrix are highly clustered near 1. By the analysis on the rate of convergence of preconditioned CG for highly clustered spectrum given in \cite{Axelsson1986}, the preconditioned matrix can be regarded as having an "efficient" condition number $b/a$, where $[a,b]$ is the closed interval in which most of the eigenvalues are clustered. Therefore, a fast convergence rate for preconditioned CG is expected due to the cluster of eigenvalues at $1$ and the small number of outliers.

\begin{table}[h]
\caption{Numbers of iterations with CG for $e^{A_n[f]}$ with the generating ${f(\theta) = \frac{4}{3}\theta\cos(\theta)}$.}
\label{tab:ExpTCG_ch2_hon16}
\begin{center}
{\small
\begin{tabular}{|l|c|c|}\hline
 $n $ & $I_n$  & $e^{c_n[f]}$
\\
\hline
128 & 224 & 20 \\
256 & 325 & 21 \\
512 & 414 & 25 \\
1024 & 491 & 26 \\
\hline
\end{tabular}}
\end{center}
\end{table}

\begin{figure}[htbp]
\centering
\subfloat[]{\includegraphics[scale=0.7]{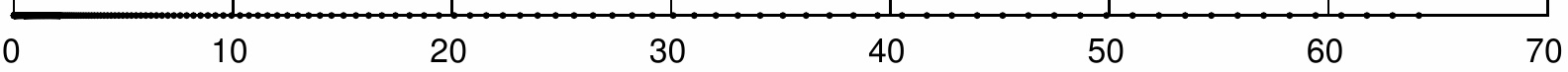}}\\
\subfloat[]{\includegraphics[scale=0.7]{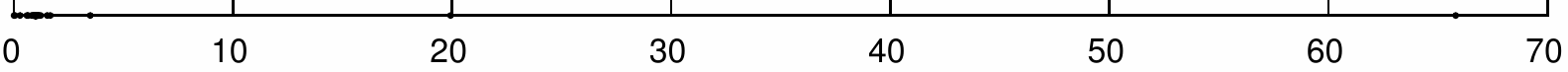}}\\
\subfloat[]{\includegraphics[scale=0.7]{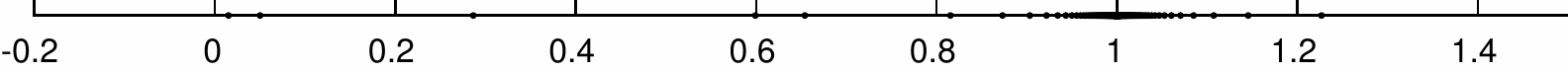}}
\caption{The spectrum of {(a)} $e^{A_n[f]}$ and that of {(b)} ${(e^{c_n[f]})}^{-1}e^{A_n[f]}$. {(c)} The zoom-in spectrum of {(b)}. $A_n[f]$ is generated by ${f(\theta) = \frac{4}{3}\theta\cos(\theta)}$ and $n=512$.}
\label{fig:exp_512_CG_ch2_hon16}
\end{figure}


In Figure \ref{fig:exp_n_PCG_ch2_hon16}, we further show the spectrum of $e^{A_n[f]}$ before and after applying the preconditioner $e^{c_n[f]}$ with different $n$. We observe that the highly clustered spectra seem independent of $n$.

%
%
%

\begin{figure}[htbp]
\centering
{(a)}~\includegraphics[scale=0.7]{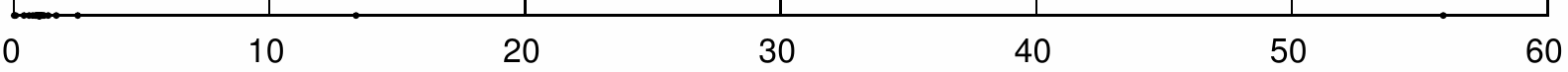}\\
\vspace{0.5cm}
\includegraphics[scale=0.7]{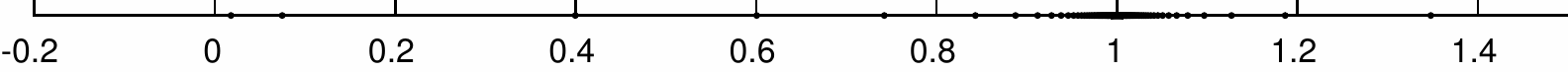}\\
\vspace{2cm}
{(b)}~\includegraphics[scale=0.7]{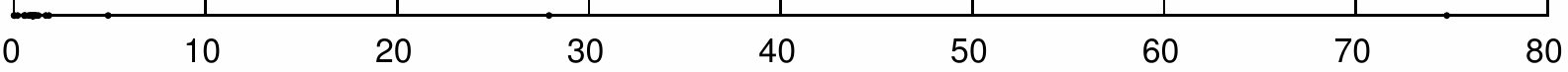}\\
\vspace{0.5cm}
\includegraphics[scale=0.7]{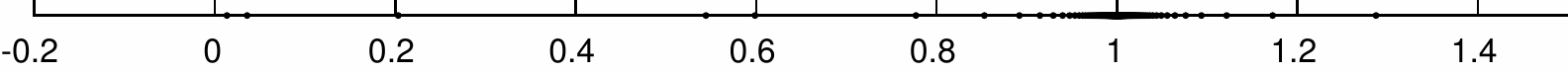}\\
\vspace{2cm}
{(c)}~\includegraphics[scale=0.7]{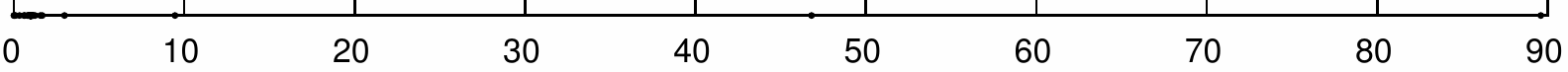}\\
\vspace{0.5cm}
\includegraphics[scale=0.7]{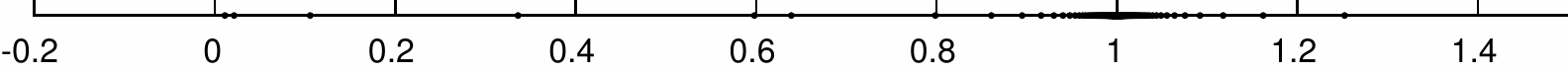}\\

\caption{The spectrum of $e^{A_n[f]}$ and that of ${(e^{c_n[f]})}^{-1}e^{A_n[f]}$ {(a)} $n=256$, {(b)} $n=1024$ or {(c)} $n=4096$. $A_n[f]$ is generated by ${f(\theta) = \frac{4}{3}\theta\cos(\theta)}$.}
\label{fig:exp_n_PCG_ch2_hon16}
\end{figure}

\newpage

Example 2: Table \ref{tab:ExpNormal_ch2_hon16} {(a)} and {(b)} show the numerical results using CG and GMRES for the normal equations matrices of $e^{A_n[f]}$ with $A_n[f]$ generated by ${f(\theta) = 2\theta\cos(\theta)+\theta \mathbf{i}}$, respectively. Again, we observe that the preconditioners are efficient for speeding up the rate of convergence.

\begin{table}[h]
\caption{Numbers of iterations with {(a)} CG for ${(e^{A_n[f]})^*e^{A_n[f]}}$ and {(b)} GMRES for $e^{A_n[f]}$ with $A_n[f]$ generated by ${f(\theta) = 2\theta\cos(\theta)+\theta \mathbf{i}}$.}
\label{tab:ExpNormal_ch2_hon16}
\begin{center}
{{(a)}~\small
\begin{tabular}{|l|c|c|}\hline
 $n $ & $I_n$  & Preconditioner

\\
\hline
128 & 14219 & 75\\
256 & 78645 & 96\\
512 & $>$100000 & 145\\
1024 & $>$100000 & 110\\
\hline
\end{tabular}}\\
\vspace{0.5cm}

{{(b)}~\small
\begin{tabular}{|l|c|c|}\hline
 $n $ & $I_n$  & Preconditioner 
\\

\hline
128 & 128 & 21\\
256 & 248 & 23\\
512 & 477 & 25\\
1024 & 891 & 27\\
\hline
\end{tabular}}
\end{center}
\end{table}

\newpage

Example 3: We next consider the matrix sine functions. Table \ref{tab:SinMINRES_ch2_hon16} {(a)} shows the numerical results using MINRES for $\sin{A_n[f]}$ with $A_n[f]$ generated by ${f(\theta) = -(\frac{\theta^2}{2\pi}+\frac{1}{10^3})}$. As the matrix in this case is symmetric negative definite, we also show numerical results using CG as a comparison in Table \ref{tab:SinMINRES_ch2_hon16} {(b)}.

\begin{table}[h]
\caption{Numbers of iterations with {(a)} MINRES and {(b)} CG for $\sin{A_n[f]}$ with $A_n[f]$ generated by ${f(\theta) = -(\frac{\theta^2}{2\pi}+\frac{1}{10^3})}$.}
\label{tab:SinMINRES_ch2_hon16}
\begin{center}
{(a)}~{\small
\begin{tabular}{|l|c|c|}\hline
 $n $ & $I_n$  & $|\sin{c_n[f]}|$ 
\\

\hline
128 & 138 & 15 \\
256 & 227 & 14 \\
512 & 238 & 11 \\
1024 & 243 & 10 \\
\hline
\end{tabular}}\\
\vspace{0.5cm}

{(b)}~{\small
\begin{tabular}{|l|c|c|}\hline
 $n $ & $I_n$  & $\sin{c_n[f]}$ 
\\
\hline
128 & 138 & 16\\ 
256 & 235  & 14\\ 
512 & 253 & 11\\ 
1024 & 255 & 10\\
\hline
\end{tabular}}

\end{center}
\end{table}

\newpage

Example 4: Table \ref{tab:SinNormal_ch2_hon16} shows the numerical results for $(\sin{A_n[f]})^*\sin{A_n[f]}$ with $A_n[f]$ generated by $f(\theta) = -(\frac{\theta^2}{2\pi}\mathbf{i} + \frac{1}{10^3})$. Since the normalised matrices are highly ill-conditioned, CG without preconditioner requires large numbers of iteration to get the solutions to the desired tolerance. However, the numbers of iterations are reduced significantly with our proposed preconditioner.

\begin{table}[h]
\caption{Numbers of iterations with {(a)} CG for $(\sin{A_n[f]})^*\sin{A_n[f]}$ and {(b)} GMRES for $\sin{A_n[f]}$ with $A_n[f]$ generated by ${f(\theta) = -(\frac{\theta^2}{2\pi}\mathbf{i} + \frac{1}{10^3})}$.}
\label{tab:SinNormal_ch2_hon16}
\begin{center}
{(a)}~{\small
\begin{tabular}{|l|c|c|}\hline
 $n $ & $I_n$  & Preconditioner
\\
\hline
128 & 1094 & 31\\
256 & 3238 & 27\\
512 & 4844 & 22\\
1024 & 10152 & 16\\
\hline
\end{tabular}}\\
\vspace{0.5cm}

{(b)}~{\small
\begin{tabular}{|l|c|c|}\hline
 $n $ & $I_n$  & Preconditioner 
\\
\hline
128 & 128 & 16\\ 
256 & 255 & 17\\
512 & 384 & 13\\ 
1024 & 463 & 10\\
\hline
\end{tabular}}
\end{center}
\end{table}

\newpage

Example 5: Lastly, we consider the matrix cosine functions. Table \ref{tab:CosMINRES_ch2_hon16} and \ref{tab:CosNormal_ch2} show the numerical results for symmetric matrix $\cos{A_n[f]}$ with $A_n[f]$ generated by $f(\theta)= (\frac{\pi}{2}-\frac{1}{10^4})\cos{(\theta^2)}-\frac{\pi}{4}$ and for $(\cos{A_n[f]})^*\cos{A_n[f]}$ with $A_n[f]$ generated by $f(\theta) = (\frac{\pi}{2}-\frac{1}{10^4})\cos{(\theta^2})+\frac{\theta}{\pi}\mathbf{i}$, respectively. In Figure \ref{fig:cos_512_MINRES_ch2_hon16} {(a)} and {(b)}, we also show the spectrum of the matrices before and after applying the preconditioner $|\cos{A_n[f]}|$. In the zoom-in spectrum shown in Figure \ref{fig:cos_512_MINRES_ch2_hon16} {(c)}, we observe that the eigenvalues of the preconditioned matrix are mostly $\pm1$. We conclude that our proposed preconditioners appear effective for these systems defined by matrix cosine functions of Toeplitz matrices.

\begin{table}[h]
\caption{Numbers of iterations with MINRES for $\cos{A_n[f]}$ with $A_n[f]$ generated by ${f(\theta) = (\frac{\pi}{2}-\frac{1}{10^4})\cos{(\theta^2)}-\frac{\pi}{4}}$.}
\label{tab:CosMINRES_ch2_hon16}
\begin{center}
{\small
\begin{tabular}{|l|c|c|}\hline
 $n $ & $I_n$  & $|\cos{c_n[f]}|$
\\
\hline
128 & 77 & 30 \\
256 & 139 & 36 \\
512 & 261 & 38 \\
1024 & 506 & 42 \\
\hline
\end{tabular}}
\end{center}
\end{table}

\begin{table}[h]
\caption{Numbers of iterations with {(a)} CG for $(\cos{A_n[f]})^*\cos{A_n[f]}$ and {(b)} GMRES for $\cos{A_n[f]}$ with $A_n[f]$ generated by ${f(\theta) = (\frac{\pi}{2}-\frac{1}{10^4})\cos{(\theta^2})}+\frac{\theta}{\pi}\mathbf{i}$.}
\label{tab:CosNormal_ch2}
\begin{center}
{(a)}~{\small
\begin{tabular}{|l|c|c|}\hline
 $n $ & $I_n$  & Preconditioner 
\\
\hline
128 & 191 & 21 \\
256 & 435 & 22 \\
512 & 973 & 22 \\
1024 & 2092 & 23 \\
\hline
\end{tabular}}\\
\vspace{0.5cm}

{(b)}~{\small
\begin{tabular}{|l|c|c|}\hline
 $n $ & $I_n$  & Preconditioner 
\\
\hline
128 & 126 & 16\\
256 & 249 & 16\\
512 & 492 & 17\\
1024 & 972 & 17\\
\hline
\end{tabular}}

\end{center}
\end{table}

\begin{figure}[htbp]
\centering
\subfloat[]{\includegraphics[scale=0.7]{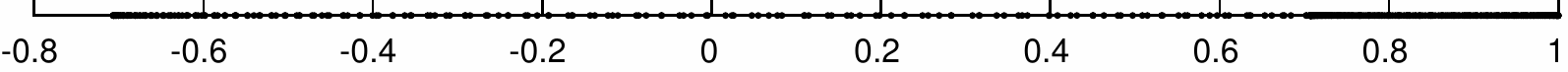}}\\
\subfloat[]{\includegraphics[scale=0.7]{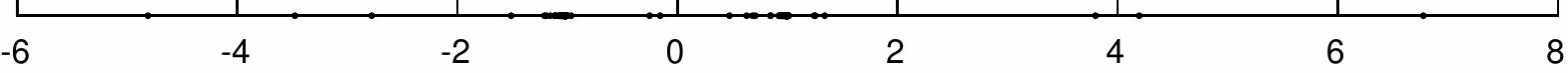}}\\
\subfloat[]{\includegraphics[scale=0.7]{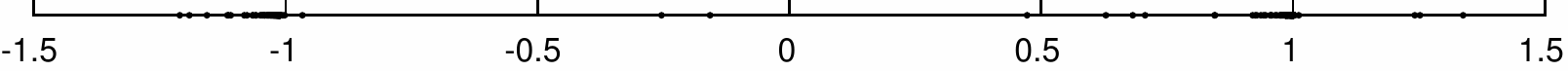}}\\
\caption{The spectrum of {(a)} $\cos{A_n[f]}$ and that of {(b)} ${|\cos{c_n[f]}|}^{-1}\cos{A_n[f]}$. {(c)} The zoom-in spectrum of {(b)}. $A_n[f]$ is generated by ${f(\theta) = (\frac{\pi}{2}-\frac{1}{10^4})\cos{(\theta^2)}-\frac{\pi}{4}}$ and $n=512$.}
\label{fig:cos_512_MINRES_ch2_hon16}
\end{figure}

%

\medskip

\bibliographystyle{plain}
\bibliography{SeanReferences}

\end{document}